\documentclass[12pt]{article}

\usepackage{packages}
\usepackage{environments}
\usepackage{commands}

\begin{document} 

\input{header}
\input{section_1_intro}
\ProvidesFile{section_2_prerequisites.tex}[section 2 prerequisites]

\section{Prerequisites}

All ring under consideration are associative with 1 not equal to 0.

\begin{lemma}[see \cite{Lam}]
For any non-zero ring $R$, the following statements are equivalent:

\begin{enumerate}[a)]
\item $R$ has a unique maximal left ideal;

\item $R$ has a unique maximal right ideal;

\item a factor ring $R$ by its radical is a skew field;

\item all non-units of $R$ form an ideal;

\item all non-units of $R$ form an additive group.
\end{enumerate}
\end{lemma}

\begin{definition}
A ring satisfying any condition from the previous lemma is called \it{local}.
\end{definition}

Let $R$ be a local ring, and let $V$ be a free $R$-module of finite rank, then, by theorem 20.13 (2) from \cite{Lam}, the rank of $V$ is well-defined. Consider $\GL(V)$ --- a group of all invertible $R$-linear endomorphisms on $V$. Fix a basis in $V$, then $\GL(V)$ is a group of all non-degenerate $n\times n$-matrices over $R$, i.~e. $\GL_n(R)$, where $n$ is the rank of $V$.

Let us briefly recall the concept of elementary equivalence.

\begin{definition}
Formula $\varphi$ of signature $\Sigma$ is called \emph{universal} (\emph{existential}), if its prenex normal form is written as
$$
Q_1 x_1\ldots Q_n x_n \psi(x_1,\ldots,x_n),
$$
where $Q_1=\dots= Q_n=\forall \text{ }(Q_1=\dots=Q_n=\exists)$, and $\psi$ is quantifier-free.
\end{definition}
 
\begin{definition}
Two algebraic systems $\mathfrak{A}$ and $\mathfrak{B}$ of signature $\Sigma$ are called {\it universally equivalent} ({\it existentially equivalent}), if for every universal (existential) predicate $\varphi$ of signature $\Sigma$ the following holds true:
$$
\mathfrak{A}\models\varphi\Longleftrightarrow\mathfrak{B}\models\varphi.
$$
The set of all universal (existential) predicates $\{\varphi\mid\mathfrak{A}\models\varphi\}$ of signature $\Sigma$ is called the {\it universal} ({\it existential}) {\it theory} of system $\mathfrak{A}$, it is denoted by $\mbox{Th}_{\forall}(\mathfrak{A})$ ($\mbox{Th}_{\exists}(\mathfrak{A})$). Hence, $\mathfrak{A}\equiv_{\forall}\mathfrak{B}\Leftrightarrow\mbox{Th}_{\forall}(\mathfrak{A})=\mbox{Th}_{\forall}(\mathfrak{B})\Leftrightarrow\mbox{Th}_{\exists}(\mathfrak{A})=\mbox{Th}_{\exists}(\mathfrak{B})\Leftrightarrow\mathfrak{A}\equiv_{\exists}\mathfrak{B}$. The latter follows from the quantifiers dependence.
\end{definition}

We will use the following {\it criterion for universal equivalence}: two algebraic systems of the same finite signature are universally equivalent iff every finite submodel of the first system has an isomorphic submodel in the second system and vice versa.

In this paper we prove the following theorem.

\begin{theorem}\label{maintheoremlocal}
Let $R_1$ and $R_2$ be local (non necessarily commutative) rings with $1/2$ and let groups $\GL_n(R_1)$ and $\GL_m(R_2)$ admit automorphisms, defined by formula $\varphi(x) = (x^T)^{-1}$ (inverse-transpose automorphism). 

Then groups $\GL_n(R_1)$ and $\GL_m(R_2)$ $(m, n\geqslant 3)$ are universally equivalent if and only if $n=m$ and rings $R_1$ and $R_2$ are universally equivalent.
\end{theorem}

The easier part (universal equivalence of rings imply universal equivalence of linear groups) is similar to the proof for fields, see \cite{BuninaKaleeva}. Note that the statement holds true even in case $n = m < 3$ and groups do not admit the inverse-transpose automorphism.

Next we prove the harder part: universal equivalence of linear groups implies the coincidence of size and universal equivalence of corresponding rings.

Note that if linear groups are universally equivalent, then rings $R_1$ and $R_2$ are both either finite of infinite. In case of finite rings, universal equivalence is the same as isomorphism, hence the statement follows from  paper \cite{GolubchikMichalev}. It is proven that the following statements are equivalent in case natural $n, m\geqslant 3$:

\begin{enumerate}
\item $n=m$ and $R_1\cong R_2$,

\item $\GL_n(R_1)\cong \GL_m(R_2)$,
\end{enumerate}
where $R_1$ and $R_2$ are local associative rings with $1/2$.

Next rings $R_1$ and $R_2$ are considered to be infinite.
\ProvidesFile{section_3_main_theorem.tex}[section 3 main theorem]

\section{The main statement proof}

First we study the structure of involutions in a general linear group over a non-commutative ring. We will follow the article \cite{PomfretMcDonald}, where similar consideration was provided for local commutative rings.

\begin{lemma}\label{involution}
Let $R$ be a local ring with $1/2$, and $I\in \GL_n(R)$ be an element, such that $I^2 = E$. Then there exist a basis of left $R$-module $V$ of rank $n$, such that $I$ is represented by a matrix $\diag[1, \dots, 1, -1, \dots, -1]$ in this basis, where the number of ones on the diagonal is $t$, $-1$ --- is $s$, $t + s = n$ and  $t, s$ are uniquely defined.
\end{lemma}

\begin{proof}
Consider the following submodules:
$$
N(I) = \{v\in V | I(v) = - v\}, \text{ }
P(I) = \{v \in V | I(v) = v\}.
$$

Assume $u\in N(I)\cap P(I)$, then $I(u)=u=-u$, which, combined with the existence of $1/2$ in the ring, implies $u=0$.

For an arbitrary element $u$ in $V$ we have: 
$$
u = \frac{1}{2}(u - I(u)) + \frac{1}{2}(u + I(u)),
$$ 
where $u - I(u)\in N(I)$, $u + I(u)\in P(I)$. Hence, $V = N(I)\oplus P(I)$. 

The only thing left to do now is to make a few comments on the structure of submodules $N(I)$ and $P(I)$. Firstly, $N(I)$ and $P(I)$ are projective modules, because they are direct summands of a free module. Secondly, since projectve modules over a local (not necessarily commutative) ring are free (\cite{Lam}, theorem 19.29), then $N(I)$ and $P(I)$ are free modules, moreover their ranks are finite and well-defined. In a basis consisting of bases of $P(I)$ and $N(I)$, $I$ has the desired form.
\end{proof}

\begin{lemma}
Let $R$ be a local ring with $1/2$. Consider elements $I_1, \dots, I_m\in \GL_n(R)$, such that $I_j^2=E$ and $I_j I_k = I_k I_j$ for all $j, k = 1, \dots, m$. Then there exists a basis in the left $R$-module of rank  $n$, such that all $I_1, \dots, I_m$ are diagonal with $\pm 1$ on their diagonals.
\end{lemma}

\begin{proof}
First consider a pair of two commuting operators of order 2: $I_1, I_2$. By the lemma \ref{involution}, we can fix such a basis, that $I_1=\diag[1, \dots, 1, -1, \dots, -1]$ ($t$ ones and $s$ minus ones). Combined with the fact that $2$ is invertible, the equality $I_1 I_2 = I_2 I_ 1$ implies that, in the previously fixed basis, $I_2$ consists of two blocks of size $t\times t$ and $s\times s$ on the diagonal.

Consider $v\in P(I_1)$, then $I_1 I_2(v) = I_2 I_1 (v)=I_2(v)$, hence $I_2(v)\in P(I_1)$. It means that submodule $P(I_1)$ is invariant with respect to $I_2$, the same is true for $N(I_1)$. Now it is sufficient to change bases in $P(I_1)$ and $N(I_1)$, in such a manner that operator $I_2$ has the desired form; this does not affect $I_1$.

For more then two operators we can build the same proof by induction, because for every additional operator all intersections of previous invariant submodules are invariant with respect to this new operator.
\end{proof}

\begin{corollary}
Let $R$ be a local ring with $1/2$. The biggest cardinality of a subset of $\GL_n(R)$, consisting of pairwise commuting involutions, is equal to $2^n - 1$.
\end{corollary}

\begin{corollary}
Let $R_1$ and $R_2$ be local rings with $1/2$, let $m$ and $n$ be natural numbers, and let groups $\GL_n(R_1)$ and $\GL_m(R_2)$ be universally equivalent. Then $n=m$.
\end{corollary}

From here on we will assume that $n>2$.

Let's fix some maximal subset of pairwise commuting involutions and a basis, such that all involutions are diagonal in this basis. We call this set of involutions $\EuScript{MI}$. Note that $\EuScript{MI}$ is partitioned into conjugate classes, and each conjugate class consists of elements with the same amount of $-1$ on the diagonal. Among all the classes, $\{-E\}$ has the smallest cardinality, next classes in the order of increasing cardinality are: two $n$-element conjugate classes, consisting of matrices with exactly one or $n-1$ minus ones on the diagonal. This conjugate classes can be singled out by an existential formula. Note that a formula, based on counting elements in a conjugate class, describes any of these two classes, but does not specify which one. We call $\EuScript{I}_1$ one of these classes (does not matter, which one). Denote elements of $\EuScript{I}_1$ in the following way: $I_k$ is a matrix with a different element at the $k$-th place on the diagonal.

For any submodel isomorphism with $\EuScript{MI}$ in the domain, the image of this subset in some basis consists of matrices of the same form as $\EuScript{MI}$ in our previously fixed basis. In case the domain contains any diagonal matrix, its image is diagonal as well, because diagonal matrices are exactly the matrices which commute with all elements of $\EuScript{MI}$.

Later, the following lemma will be useful.

\begin{lemma}
Let $R$ be a local ring with $1/2$. Consider an element $x\in R$, s.~t. $x^2 = 1$. Then $x = \pm 1$.
\end{lemma}

\begin{proof}
Note that $\frac{1}{2}(x+1)$ is an idempotent. Since ring $R$ is local, this idempotent is trivial: $\frac{1}{2}(x+1) = 0\text{ or } 1$. In the former case $x = -1$, in the latter $x = 1$.
\end{proof}

Now we are ready to prove that some special kinds of matrices preserve their form under submodel isomorphisms.

\begin{lemma}\label{transp}
Let $R_1$ and $R_2$ be infinite local rings with $1/2$ and $\GL_n(R_1)\equiv_{\forall}\GL_n(R_2)$ $(n > 2)$. Let $M_1$ be an arbitrary finite submodel of $\GL_n(R_1)$, s.~t. it contains $\EuScript{MI}$ and matrices, conjugating elements of $\EuScript{MI}$, and let $M_2$ be a finite submodel $\GL_n(R_2)$, which is isomorphic to $M_1$ (it exists due to the universal equivalence criterion).

Then for every isomorphism  
$$
\Phi\colon M_1 \rightarrow M _2
$$
the image of $\sigma_{12}=E-E_{11}-E_{22}+E_{12}+E_{21}$ is the following matrix 
$$
\begin{pmatrix}
0 & 1 & 0 & \dots & 0\\
1 & 0 & 0 & \dots & 0\\
0 & 0 & \pm 1 & &\\
\vdots & \vdots &  & \ddots & \\
0 & 0 & & & \pm 1\\
\end{pmatrix}
$$ with respect to a basis, s.~t. all matrices from $\EuScript{MI}$ do not change their form.
\end{lemma}

\begin{proof}
Denote $\Phi(\sigma_{12})=(d_{ij})$.

Since $\sigma_{12}$ satisfies the following conditions:
\begin{align}
\forall k > 2 (\sigma_{12} I_k=I_k\sigma_{12}),\\
\sigma_{12}^2=E,\\
\sigma_{12}I_1\sigma_{12}=I_2,
\end{align}
so does its image.

From the first condition we have 
$$
\Phi(\sigma_{12})=
\begin{pmatrix}
d_{11} & d_{12} & 0 & \dots & 0\\
d_{21} & d_{22} & 0 & \dots & 0\\
0 & 0 & d_{33}& \dots & 0\\
\vdots & \vdots & \vdots & \ddots & \vdots\\
0 & 0 & 0 & \dots & d_{nn}\\
\end{pmatrix},
$$  
using the second we get $d_{kk} = \pm 1$ for $3 \leqslant k \leqslant n$.

From the second and third conditions for the left upper corner (up to multiplication by $-E$) we get:
\begin{multline*}
\begin{pmatrix}
	d_{11}^2+d_{12}d_{21} & d_{11}d_{12}+d_{12}d_{22} \\
	d_{21}d_{11}+d_{22}d_{21} & d_{21}d_{12}+d_{22}^2 \\
\end{pmatrix}=
\begin{pmatrix}
	1 & 0\\
	0 & 1
\end{pmatrix},\\
\begin{pmatrix}
d_{11} & d_{12}\\
d_{21} & d_{22}
\end{pmatrix}
\begin{pmatrix}
-1 & 0\\
0 & 1
\end{pmatrix}
\begin{pmatrix}
d_{11} & d_{12}\\
d_{21} & d_{22}
\end{pmatrix}=\\
=\begin{pmatrix}
-d_{11}^2+d_{12}d_{21} & -d_{11}d_{12}+d_{12}d_{22}\\
-d_{21}d_{11}+d_{22}d_{21} & -d_{21}d_{12}+d_{22}^2
\end{pmatrix}=
\begin{pmatrix}
1 & 0\\
0 & -1
\end{pmatrix}.
\end{multline*}
Summing up $(1, 1)$-elements and subtracting $(4, 4)$-elements, we obtain: $2 d_{12} d_{21} = 2$ and $2 d_{21} d_{12} = 2$, hence $d_{12}=\alpha$, $d_{21} = 1/\alpha$, where $\alpha$ is a revertible element of $R_2$.

Summing up $(1, 2)$-elements, we have $2d_{12}d_{22} = 0$. Since $2d_{12}$ is invertible, we get $d_{22}=0$, the same way we get $d_{11} = 0$.

Now we change basis using change of basis matrix 
$$
\diag[1/\alpha, 1, \dots, 1],
$$
which commutes with every element of $\EuScript{MI}$, and get the desired form for $\Phi(\sigma_{12})$. 
\end{proof}

\begin{corollary}
Let $R_1$ and $R_2$ be infinite local rings with $1/2$ and $\GL_n(R_1)\equiv_{\forall}\GL_n(R_2)$ $(n > 2)$. And let $M_1$ be an arbitrary finite submodel of $\GL_n(R_1)$, containing $\EuScript{MI}$ and conjugating matrices, and let $M_2$ be a finite submodel of $\GL_n(R_2)$, which is isomorphic to $M_1$.

Then for every isomorphism  
$$
\Phi\colon M_1 \rightarrow M _2
$$
the image of $\sigma_{23}=E-E_{22}-E_{33}+E_{23}+E_{32}$ has the following form 
$$
\begin{pmatrix}
\pm 1 & 0 & 0 & 0 & \dots & 0\\
0 & 0 & 1 & 0 & \dots & 0\\
0 & 1 & 0 & 0 & \dots & 0\\
0 & 0 & 0 &  \pm 1 & &\\
\vdots & \vdots & \vdots &  & \ddots & \\
0 & 0 & 0 & & & \pm 1\\
\end{pmatrix}
$$ in a basis, in which matrices from $\EuScript{MI}$ do not change their form.	
\end{corollary}

\begin{proof}
Similar to the previous lemma, 
$$\Phi(\sigma_{23})=
\begin{pmatrix}
\pm 1& 0&0&0&\dots&0\\
0&0 & \beta & 0 & \dots & 0\\
0&\frac{1}{\beta} & 0 & 0 &\dots & 0\\
0&0 & 0 & \pm 1 & &\\
\vdots& \vdots & \vdots &&\ddots&\\
0&0 & 0 &&& \pm 1
\end{pmatrix},
$$
where $\beta$ is an invertible element of $R_2$. Now we apply change of basis matrix $\diag[1, 1, \beta, 1, \dots, 1]$, which commutes with all elements of  $\EuScript{MI}$ and $\sigma_{12}$.
\end{proof}

\begin{lemma}\label{thetransvection}
Let $R_1$ and $R_2$ be infinite local rings with $1/2$ and $\GL_n(R_1)\equiv_{\forall}\GL_n(R_2)$ $(n > 2)$. And let $M_1$ be an arbitrary finite submodel of $\GL_n(R_1)$, which contains finite set of matrices from the previous lemma and corollary and $\diag[2, 1, 1/2, 1,\dots, 1]$, and let $M_2$ be a finite submodel of $\GL_n(R_2)$, which is isomorphic to $M_1$. Then for every isomorphism  
$$
\Phi\colon M_1 \rightarrow M _2
$$
the image of $E+ E_{12}$ equals $E+E_{12}$ or $E-E_{21}$ in a basis, in which matrices from $\EuScript{MI}$, $\sigma_{12}$ and $\sigma_{23}$ do not change their form.
\end{lemma}

\begin{proof}
Similar to the proof of lemma \ref{transp}, 
$$
\Phi(E+\alpha E_{12})=
\begin{pmatrix}
t_{11} & t_{12} & 0 & \dots & 0\\
t_{21} & t_{22} & 0 & \dots & 0\\
0 & 0 & t_{33}& \dots & 0\\
\vdots & \vdots & \vdots & \ddots & \vdots\\
0 & 0 & 0 & \dots & t_{nn}\\
\end{pmatrix}.
$$

Since $D = \diag[2, 1, 1/2, 1,\dots, 1]$ commutes with all matrices from $\EuScript{MI}$, $$\Phi(\diag[2, 1, 1/2, 1,\dots, 1])=\diag[\alpha_1, \dots, \alpha_n],
$$
where all $\alpha_i$ are invertible.

Here are all the necessary equalities, true in $M_1$, which we will use in the following proof $M_1$ (similar equalities hold true for the images in $M_2$):
\begin{align}
D(E+\alpha E_{12})D^{-1} = (E+\alpha E_{12})^2, \label{diag}\\
(I_2(E+E_{12}))^2 = E, \label{units} \\
(I_2\sigma_{12}(E+E_{12}))^3 = E,\label{rockstar}\\
\sigma_{23}(E+E_{12})\sigma_{23} \text{ commutes with } E+E_{12}.\label{sigma}
\end{align}

\begin{remark}
From $(\ref{rockstar})$, in particular, it follows that $\Phi(I_2)$ has only one $-1$ at the second place on the diagonal.
\end{remark}

From (\ref{diag}) we obtain: $t_{kk} = t_{kk}^2$ and $t_{kk}$ are invertible for all $k = 3, \dots, n$, since they are only non-zero elements in corresponding rows and columns, hence $t_{kk} = 1$ for all $k = 3, \dots, n$.

For the left-upper corner conditions (\ref{diag}) -- (\ref{sigma}) yield:
\begin{gather}
\begin{pmatrix}
t_{11}^2 + t_{12}t_{21} & t_{11}t_{12} + t_{12}t_{22}\\
 t_{21}t_{11} + t_{22}t_{21} & t_{21}t_{12} + t_{22}^2
\end{pmatrix} = 
\begin{pmatrix}
\alpha_1 t_{11}\alpha_1^{-1} & \alpha_1 t_{12}\alpha_2^{-1}\\
\alpha_2 t_{21}\alpha_1^{-1} & \alpha_2 t_{22}\alpha_2^{-1}
\end{pmatrix},\label{diagdash}\\
\begin{pmatrix}
t_{11}^2 - t_{12}t_{21} & t_{11}t_{12} - t_{12}t_{22}\\
-t_{21}t_{11} + t_{22}t_{21} & -t_{21}t_{12} + t_{22}^2
\end{pmatrix} =
\begin{pmatrix}
1 & 0 \\
0 & 1
\end{pmatrix}, \label{unitsdash}\\
\begin{pmatrix}
t_{21}^3 - t_{22}t_{11}t_{21} - t_{21}t_{22}t_{11} + t_{22}t_{12}t_{11} &
t_{21}^2t_{22} - t_{22}t_{11}t_{22} -t_{21}t_{22}t_{12} + t_{22}t_{12}^2\\
-t_{11}t_{21}^2+t_{12}t_{11}t_{21}+t_{11}t_{22}t_{11}-t_{12}^2t_{11} &
-t_{11}t_{21}t_{22}+t_{12}t_{11}t_{22}+t_{11}t_{22}t_{12}-t_{12}^3
\end{pmatrix}=\begin{pmatrix}
1 & 0\\
0 & 1
\end{pmatrix},\label{rockstardash}\\
\begin{pmatrix}
t_{11}^2 & t_{11}t_{12} & t_{12}\\
t_{21} & t_{22} & 0\\
t_{21}t_{11} & t_{21}t_{12} & t_{22}
\end{pmatrix} = 
\begin{pmatrix}
t_{11}^2 & t_{12} & t_{11}t_{12}\\
t_{21}t_{11} & t_{22} & t_{21}t_{12}\\
t_{21} & 0 & t_{22}
\end{pmatrix}.\label{sigmadash}
\end{gather}

From (\ref{sigmadash}) we get $t_{21}t_{12} = 0$, plug it in (\ref{unitsdash}) into (2, 2)-element and obtain $t_{22}^2 = 1$, hence $t_{22} = \pm 1$ which belongs to the center of ring $R_2$. Examining (2, 2)-element of (\ref{diagdash}) we get: $t_{22}^2 = \alpha_2 t_{22}\alpha_2^{-1} = t_{22}$, then $t_{22} = 1$.

Consider two cases: 1) $t_{12}$ is invertible, 2) $t_{11}$ is invertible (at least one of these holds true, because $\Phi(E+E_{12})$ is invertible).

{\it Case 1:}  $t_{12}$ is invertible.

We use (\ref{unitsdash}) again, considering (1, 2)-elements we obtain $(t_{11} - 1)t_{12} = 0$, hence $t_{11} = 1$. Then at place(1, 1) in (\ref{diagdash}) we have $1 + t_{12}t_{21} = 1$ and $t_{21} = 0$.

To prove that $t_{12} = 1$, we use (\ref{rockstardash}), where (1, 1)-element equals 
$1 = t_{21}^3 - t_{22}t_{11}t_{21} - t_{21}t_{22}t_{11} + t_{22}t_{12}t_{11} = t_{12}$.

{\it Case 2:}  $t_{11}$ is invertible.
Plug $t_{22} = 1$ and $t_{21}t_{12} = 0$ in (1, 2)-element of (\ref{rockstardash}): 
$$
t_{21}^2t_{22} - t_{22}t_{11}t_{22} -t_{21}t_{22}t_{12} + t_{22}t_{12}^2=t_{21}^2-t_{11}+t_{12}^2 = 0,
$$
since $t_{11}$ is invertible, equality $t_{21}^2 + t_{12}^2 = t_{11}$ inplies, that at least one of elements $t_{12}, t_{21}$ is invertible. If $t_{12}$ is invertible, the lemma is proven. Assume now, $t_{21}$ is invertible. Then from (\ref{sigmadash}) it follows that $t_{11} = 1$ and $t_{12} = 0$ (elements (3, 1) and (3, 2) correspondingly). Let us find $t_{21}$. At places (1,1) and (1,2) in (\ref{rockstardash}) we have the following: 
$$
t_{21}^3 - 2t_{21} = 1, \\
- t_{21}^2 + 1 = 0,
$$
then $t_{21} = -1$. Q.~e.~d.
\end{proof}

Under the conditions of the previous lemma, additionally assume that groups $\GL_n(R_1)$ and $\GL_n(R_2)$ admit the inverse-transpose automorphism. Note that this automorphism acts identically on diagonal matrices and permutation matrices $\sigma_{12}, \sigma_{23}$. In case $\Phi(E+E_{12}) = E - E_{21}$, we can apply the inverse-transpose automorphism to the image and obtain, that the form of $E+E_{12}$ does not change.

\begin{lemma}\label{transvections}
Let $R_1$ and $R_2$ be infinite local comuutative rings with $1/2$ and $\GL_n(R_1)\equiv_{\forall}\GL_n(R_2)$, also $\GL_n(R_1)$ and $\GL_n(R_2)$ admit the inverse-transpose automorphism. Let $M_1$ be an arbitrary finite submodel of $\GL_n(R_1)$, which contains the finite set of matrices from the previous lemmas and $\diag[2, 1, 1/2, 1,\dots, 1]$, and let $M_2$ be a finite submodel of $\GL_n(R_2)$, which is isomorphic to $M_1$. Then for every isomorphism
$$
\Phi\colon M_1 \rightarrow M _2
$$
the image of any finite set of matrices of form $E+ \alpha_i E_{12}$, $i=1, \dots, k$ is equal to a set of matrices of form $E+\beta_i E_{12}$, $i=1, \dots, k$, in a basis, in which the forms of all matrices from the previous lemmas do not change.
\end{lemma}

\begin{proof}
Similar to lemma \ref{thetransvection} (using commuting with $I_3, \dots, I_n$ and condition (\ref{diag})), we can easily show 
$$
\Phi(E+\alpha_i E_{12})=
\begin{pmatrix}
s^i_{11} & s^i_{12} & 0\\
s^i_{21} & s^i_{22} & 0\\
0 & 0 & E
\end{pmatrix}.
$$

Then using the fact that $E+E_{12}$ and $E+\alpha_i E_{12}$ commute, we get $s^i_{11}=s^i_{22}$ and $s^i_{21}=0$.

All that remains is to prove that two first diagonal elements are ones. Use the condition, which includes a diagonal matrix, for the left-upper corner and the following property of diagonal matrices
$$\Phi(\diag[2, 1, 1/2, 1, \dots, 1])=\diag[d_1, \dots, d_n],$$
we get
$$
\begin{pmatrix}
d_1s^i_{11}d_1^{-1} & d_1s^i_{12}d_2^{-1} \\
0 & d_2s^i_{11}d_2^{-1} \\
\end{pmatrix}=
\begin{pmatrix}
(s^i_{11})^2 & s^i_{11}s^i_{12} + s^i_{12} s^i_{11}\\
0 & (s^i_{11})^2
\end{pmatrix}.
$$
Note that $(I_2(E+\alpha_i E_{12}))^2 = E$, it follows that for the images the following condition holds true:
$$
\begin{pmatrix}
1 & 0\\
0 & 1
\end{pmatrix} = 
\begin{pmatrix}
s^i_{11} & s^i_{12}\\
0 & s^i_{11}
\end{pmatrix}^2=
\begin{pmatrix}
(s^i_{11})^2 & s^i_{11}s^i_{12}-s^i_{12}s^i_{11}\\
0 & (s^i_{11})^2
\end{pmatrix}.
$$
It follows, that $s^i_{11} = \pm 1$ is an element from the center, then $1 = (s^i_{11})^2 = d_1s^i_{11}d_1^{-1} = s^i_{11}$.
\end{proof}

Let us make a few remarks on the connections between the operations in local ring $R$ and group $\GL_n(R)$. Firstly,
\begin{gather*}
(E+\alpha E_{12})(E+\beta E_{12}) = E+(\alpha +\beta)E_{12},\\
\end{gather*}
secondly,
\begin{gather*}
[E+\alpha E_{12}, E+\beta E_{23}]=E+\alpha\beta E_{13},\\
\end{gather*}
additionally,
\begin{gather*}
E+\alpha E_{23}=\sigma_{12}\sigma_{23}(E+\alpha E_{12})\sigma_{23}\sigma_{12},\\
E+\alpha E_{13}=[E+E_{12}, E+\alpha E_{23}].\\
\end{gather*}

Now the main proposition follows from the lemmas similar to article \cite{BuninaKaleeva}. In more detail, let $S_1\subset R_1$ be a finite submodel in $R_1$. Our goal is to find in $R_2$ an isomorphic to $S_1$ submodel $S_2$. Let $M_1\subset\GL_n(R_1)$ be a finite submodel, which contains all matrices of form $E+\alpha E_{12}$ for all $\alpha\in S_1$, and the finite set of all auxiliary matrices from the lemmas. Since groups are universally equivalent, for $M_1$ there exists a submodel $M_2\subset\GL_n(R_2)$, isomorphic to $M_1$. Moreover, every isomorphism $\Phi\colon M_1\rightarrow M_2$ does not change the form of matrices mentioned above. Hence, in some common basis the images of matrices $E+\alpha E_{12}, E+\alpha E_{13}, E+\alpha E_{23}$ have form $E+\beta_1 E_{12}, E+\beta_2 E_{13}, E+\beta_3 E_{23}$, where $\beta_1=\beta_2=\beta_3\in R_2$. Then define $S_2$ as the set of all elements of $R_2$, which correspond triples of images $E+\alpha E_{12}, E+\alpha E_{13}, E+\alpha E_{23}$ for all $\alpha\in S_1$.

\ProvidesFile{references.tex}[references]


\begin{thebibliography}{100}
\bibitem{Lam}
Lam, T.Y. (1991). A first course in noncommutative rings. Springer.
	
\bibitem{BuninaKaleeva}
Bunina, E.I., Kaleeva, G.A. Universal Equivalence of General and Special Linear Groups Over Fields. J Math Sci 237, 387–409 (2019).

https://doi.org/10.1007/s10958-019-04165-5

\bibitem{Kaleeva}
Kaleeva, G.A. Universal Equivalence of Linear Groups Over Local Commutative Rings with 1/2. Algebra Logic 58, 313–321 (2019). 

https://doi.org/10.1007/s10469-019-09552-0

\bibitem{Maltsev}
Mal’tsev, A. I. Elementary properties of linear groups. (Russian) Nekot. Probl. Mat. Mekh., Novosibirsk, 110-132 (1961).

\bibitem{AtiyahMacdonald}
Atiyah, Michael; Macdonald, Ian G. (1969). Introduction to commutative algebra. Addison–Wesley.

\bibitem{GolubchikMichalev}
Igor Z. Golubchik, Alexander V. Mikhalev, Isomorphisms of the linear groups $GL_2(R)$ over associative rings $R$, Journal of Algebra 500, 691-708 (2018).

https://doi.org/10.1016/j.jalgebra.2017.11.025.

\bibitem{PomfretMcDonald}
Pomfret, J., McDonald, B. R. (1972). Automorphisms of $\GL_n(R)$, $R$ a Local Ring. Transactions of the American Mathematical Society, 173, 379–388. 

https://doi.org/10.2307/1996281

\end{thebibliography}
\end{document}